\documentclass[12pt,leqno]{article}

\usepackage{amsmath,amssymb,amsthm}
\usepackage[colorlinks]{hyperref}
\usepackage{cite}

\hypersetup{linkcolor=blue,citecolor=blue,urlcolor=blue}

\newcommand{\R}{\mathbb{R}}
\newcommand{\diag}{\operatorname{diag}}

\title{A spectral product formula for repunits via a tridiagonal Toeplitz similarity}
\author{Johann Verwee\\\small Independent Researcher\\\small\texttt{mverwee@gmail.com}}
\date{}

\newtheorem{theorem}{Theorem}
\newtheorem{lemma}[theorem]{Lemma}
\newtheorem{proposition}[theorem]{Proposition}

\theoremstyle{definition}

\newtheorem{remark}[theorem]{Remark}

\begin{document}
\maketitle

\begin{abstract}
For $b>0$ and $n\geqslant 1$, we consider the $n\times n$ tridiagonal matrix $V_n(b)$ with diagonal entries $b+1$, superdiagonal entries $1$, and subdiagonal entries $b$. A diagonal similarity reduces $V_n(b)$ to a symmetric tridiagonal Toeplitz matrix and hence makes its spectrum explicit. Since $\det\left(V_n(b)\right)$ equals the geometric sum $1+b+\cdots+b^{n}$, taking determinants yields a finite cosine product evaluation for this quantity. As further consequences, we derive sharp bounds from the extremal eigenvalues, write down explicit eigenvectors with respect to a natural weighted inner product, and obtain a closed formula for $V_n(b)^{-1}$.
\end{abstract}



\section{Introduction}\label{sec:intro}
For a real parameter $b>0$ and an integer $m\geqslant 1$, the repunit in base $b$ is by definition
\[
R_m(b):=1+b+\cdots+b^{m-1}=\frac{b^m-1}{b-1}\qquad (b\ne 1).
\]
We adopt the continuous extension $R_m(1):=m$, so that $R_m(b)$ is defined for all $b>0$. When $b$ is an integer with $b\geqslant 2$, this coincides with the usual base-$b$ repunits; we use the term ``repunit'' by analogy for general $b>0$.

The aim of this note is to give a short spectral proof of a finite trigonometric product evaluation for $R_m(b)$. The motivation comes in part from the recent study of higher-order repunit-type sequences in \cite{PrasadKumariMohantaMahato2025}, where closely related tridiagonal determinants already appear. Our point here is that a simple diagonal similarity reduces the relevant matrix to a symmetric tridiagonal Toeplitz model, and this immediately makes the spectrum explicit.

Fix $b>0$ and an integer $n\geqslant 1$. Consider the tridiagonal matrix
\begin{equation}\label{eq:Vdef}
V_n(b)=
\begin{pmatrix}
b+1 & 1 & 0 & \cdots & 0\\
b & b+1 & 1 & \ddots & \vdots\\
0 & b & b+1 & \ddots & 0\\
\vdots & \ddots & \ddots & \ddots & 1\\
0 & \cdots & 0 & b & b+1
\end{pmatrix}.
\end{equation}
A standard continuant computation shows that $\det(V_n(b))$ coincides with $R_{n+1}(b)$; see Lemma~\ref{lem:detrepunit}.
The key observation is that $V_n(b)$ becomes symmetric after a diagonal change of basis (Section~\ref{sec:similarity}), so its eigenvalues are explicit and multiplying them yields a clean cosine product for $R_{n+1}(b)$ (Section~\ref{sec:spectrum}). We then record quick consequences and explicit eigenvectors (Section~\ref{sec:consequences}) and end with a Chebyshev identity leading to a closed inverse formula (Section~\ref{sec:inverse}).

\section{Diagonal similarity and a weighted inner product}\label{sec:similarity}
The matrix $V_n(b)$ is generally non-symmetric when $b\ne 1$, but its super- and subdiagonal entries differ only by the constant factor $b$. A diagonal change of basis balances these off-diagonal weights and produces a symmetric tridiagonal Toeplitz model. This not only forces the spectrum to be real but also exhibits the natural weighted inner product for which $V_n(b)$ is self-adjoint.

For $b>0$, let $D=\diag\left(b^{(i-1)/2}\right)_{1\leqslant i\leqslant n}$, and set
\begin{equation}\label{eq:Tdef}
T_n(b)=
{\setlength{\arraycolsep}{4pt}\renewcommand{\arraystretch}{1.08}
\begin{pmatrix}
b+1 & \sqrt{b} & 0 & \cdots & 0\\
\sqrt{b} & b+1 & \sqrt{b} & \ddots & \vdots\\
0 & \sqrt{b} & b+1 & \ddots & 0\\
\vdots & \ddots & \ddots & \ddots & \sqrt{b}\\
0 & \cdots & 0 & \sqrt{b} & b+1
\end{pmatrix}}.
\end{equation}

\begin{lemma}[Diagonal similarity and weighted symmetry]\label{lem:similarity}
Let $b>0$ and $n\geqslant 1$.
\begin{enumerate}
\item[(a)] One has $D^{-1}V_n(b)D=T_n(b)$. In particular, $V_n(b)$ and $T_n(b)$ have the same spectrum.
\item[(b)] Let $W=D^{-2}$. Then $V_n(b)^T W= W V_n(b)$, i.e. $V_n(b)$ is self-adjoint for
$\langle x,y\rangle_W=x^T Wy$ on $\R^n$.
\end{enumerate}
\end{lemma}

\begin{proof}
(a) Since $D$ is diagonal, the diagonal entries are unchanged; for each $i$ one has $(D^{-1}V_n(b)D)_{i,i}=b+1$.
For the off-diagonal entries we compute
\[
(D^{-1}V_n(b)D)_{i,i+1}=d_i^{-1}d_{i+1}=\sqrt{b},
\qquad
(D^{-1}V_n(b)D)_{i+1,i}=d_{i+1}^{-1}\,b\,d_i=\sqrt{b},
\]
and all other entries are $0$, so $D^{-1}V_n(b)D=T_n(b)$.

(b) Since $V_n(b)=D\,T_n(b)\,D^{-1}$ by (a) and $T_n(b)^T=T_n(b)$, we have $V_n(b)^T=D^{-1}T_n(b)D$.
With $W=D^{-2}$, this gives
\[
\begin{aligned}
V_n(b)^T W
&=(D^{-1}T_n(b)D)D^{-2}
= D^{-1}T_n(b)D^{-1}\\
&=D^{-2}(DT_n(b)D^{-1})
= W V_n(b).
\end{aligned}
\]
This proves (b).
\end{proof}

\section{Spectrum and spectral product factorization}\label{sec:spectrum}
By Lemma~\ref{lem:similarity} we may replace $V_n(b)$ by the symmetric Toeplitz matrix $T_n(b)$. Its eigenvalues are classical and can be written in terms of discrete sine waves; we recall the resulting formula in Proposition~\ref{prop:eigs}. On the other hand, $\det(V_n(b))$ satisfies a short continuant recurrence and equals the geometric sum $R_{n+1}(b)$ (Lemma~\ref{lem:detrepunit}). Taking the product of eigenvalues then yields a finite cosine product evaluation for $R_{n+1}(b)$ (Theorem~\ref{thm:product}).

\begin{proposition}[Eigenvalues]\label{prop:eigs}
For $b>0$ and $n\geqslant 1$, the eigenvalues of $V_n(b)$ are
\begin{equation}\label{eq:eigs}
\lambda_k(b)=b+1+2\sqrt{b}\cos\left(\frac{k\pi}{n+1}\right),
\qquad 1\leqslant k\leqslant n.
\end{equation}
In particular, $\lambda_k(b)>0$ for every $1\leqslant k\leqslant n$.
\end{proposition}

\begin{proof}
By Lemma~\ref{lem:similarity}, $V_n(b)$ is similar to the symmetric Toeplitz tridiagonal matrix $T_n(b)$. The eigenvalue formula \eqref{eq:eigs} for $T_n(b)$ is classical; see \cite{SmithPDE}. Since $k\pi/(n+1)\in(0,\pi)$ we have $\cos(k\pi/(n+1))>-1$, and therefore
\[
\lambda_k(b)=b+1+2\sqrt{b}\cos\left(\frac{k\pi}{n+1}\right)>b+1-2\sqrt{b}=(\sqrt{b}-1)^2\geqslant 0,
\]
so $\lambda_k(b)>0$.
\end{proof}

We now connect the spectral data to repunits by evaluating $\det(V_n(b))$ via an elementary continuant recursion.

\begin{lemma}[Determinant]\label{lem:detrepunit}
For $b\ne 1$ and $n\geqslant 1$, one has $\det(V_n(b))=R_{n+1}(b)$. For $b=1$, one has $\det(V_n(1))=n+1$.
\end{lemma}

\begin{proof}
Let $\Delta_n(b)=\det(V_n(b))$. Expanding along the last row gives the continuant recurrence
\[
\Delta_n(b)=(b+1)\Delta_{n-1}(b)-b\,\Delta_{n-2}(b)\qquad (n\geqslant 2),
\]
where the minus sign comes from the cofactor of the $(n,n-1)$ entry since $(-1)^{n+(n-1)}=-1$. With $\Delta_0(b)=1$ and $\Delta_1(b)=b+1$, the same recurrence and initial values hold for $R_{n+1}(b)$ when $b\neq 1$, hence $\Delta_n(b)=R_{n+1}(b)$ for $b\neq 1$.

When $b=1$, the recurrence becomes $\Delta_n(1)=2\Delta_{n-1}(1)-\Delta_{n-2}(1)$ with $\Delta_0(1)=1$, $\Delta_1(1)=2$, whose solution is $\Delta_n(1)=n+1$.
\end{proof}

We can now combine the determinant identity with the explicit eigenvalues. Since the determinant is the product of eigenvalues (counted with multiplicity), the repunit admits a closed finite cosine product.

\begin{theorem}[Spectral product formula]\label{thm:product}
For $b>0$ and $n\geqslant 1$,
\begin{equation}\label{eq:product}
R_{n+1}(b)=\prod_{k=1}^{n}\left(b+1+2\sqrt{b}\cos\left(\frac{k\pi}{n+1}\right)\right).
\end{equation}
\end{theorem}

\begin{proof}
By Lemma~\ref{lem:similarity} and Proposition~\ref{prop:eigs},
$\det(V_n(b))=\prod_{k=1}^{n}\lambda_k(b)$.
Conclude using Lemma~\ref{lem:detrepunit}.
\end{proof}

As a simple illustration, one can reparametrize the base by $b=e^{2x}$ and rewrite the factors in hyperbolic form.

\begin{remark}[A hyperbolic cosine product]\label{rem:hyperbolic}
Let $x\in\R$ and put $b=e^{2x}$. Then \eqref{eq:product} becomes for $x\neq 0$
\[
\prod_{k=1}^{n}\left(\cosh x+\cos\left(\frac{k\pi}{n+1}\right)\right)
=
\frac{\sinh\left((n+1)x\right)}{2^{\,n}\,\sinh x}.
\]
By continuity at $x=0$ one deduces that
\[
\prod_{k=1}^{n}\left(1+\cos\left(\frac{k\pi}{n+1}\right)\right)=\frac{n+1}{2^{\,n}}.
\]
\end{remark}

\section{Consequences and explicit eigenvectors}\label{sec:consequences}

The product formula admits several immediate consequences. From the explicit expression \eqref{eq:eigs} and the fact that $k\pi/(n+1)\in(0,\pi)$ for $1\leqslant k\leqslant n$ (hence $-1<\cos(k\pi/(n+1))<1$), we obtain for each $1\leqslant k\leqslant n$ the strict two-sided estimate
\[
(\sqrt{b}-1)^2<b+1+2\sqrt{b}\cos\left(\frac{k\pi}{n+1}\right)<(\sqrt{b}+1)^2.
\]
Multiplying over $k$ and using $\det(V_n(b))=\prod_{k=1}^{n}\lambda_k(b)$ yields the bounds
\begin{equation}\label{eq:det-bounds}
\det(V_n(b))<(\sqrt{b}+1)^{2n}\qquad (b>0),
\end{equation}
and, if $b\neq 1$, also
\begin{equation}\label{eq:det-bounds-lower}
(\sqrt{b}-1)^{2n}<\det(V_n(b)),
\end{equation}
while $\det(V_n(1))=n+1$ by Lemma~\ref{lem:detrepunit}. We then write down explicit eigenvectors (and the resulting weighted orthogonality), which makes the diagonalization completely explicit. Transporting the sine basis through the conjugation in Lemma~\ref{lem:similarity} yields a convenient closed form, and the weighted self-adjointness in the same lemma forces orthogonality in the corresponding inner product.

The similarity in Lemma~\ref{lem:similarity} also yields closed eigenvectors for $V_n(b)$.

\begin{proposition}[Eigenvectors and weighted orthogonality]\label{prop:evecs}
For $1\leqslant k\leqslant n$, define $v^{(k)}\in\R^n$ by
\[
v^{(k)}_j = b^{(j-1)/2}\sin\left(\frac{jk\pi}{n+1}\right),\qquad 1\leqslant j\leqslant n.
\]
Then $V_n(b)\,v^{(k)}=\lambda_k(b)\,v^{(k)}$, where $\lambda_k(b)$ is given by \eqref{eq:eigs}.
Moreover, for $k\ne \ell$ one has $\langle v^{(k)},v^{(\ell)}\rangle_W=0$, where
$\langle\cdot,\cdot\rangle_W$ is the weighted inner product of Lemma~\ref{lem:similarity}.
\end{proposition}

\begin{proof}
Let $u^{(k)}\in\R^n$ be the sine vector $u^{(k)}_j=\sin(jk\pi/(n+1))$.
Then $T_n(b)\,u^{(k)}=\lambda_k(b)\,u^{(k)}$, and $v^{(k)}=D\,u^{(k)}$, so
$V_n(b)\,v^{(k)}=D\,T_n(b)\,u^{(k)}=\lambda_k(b)\,v^{(k)}$.

For orthogonality, let $k\ne \ell$. By Proposition~\ref{prop:eigs} the eigenvalues $\lambda_k(b)$ are pairwise distinct. Since $V_n(b)$ is self-adjoint for $\langle\cdot,\cdot\rangle_W$ (Lemma~\ref{lem:similarity}), we have
\[
\begin{aligned}
\lambda_k(b)\langle v^{(k)},v^{(\ell)}\rangle_W
&=\langle V_n(b)v^{(k)},v^{(\ell)}\rangle_W\\
&=\langle v^{(k)},V_n(b)v^{(\ell)}\rangle_W\\
&=\lambda_\ell(b)\langle v^{(k)},v^{(\ell)}\rangle_W.
\end{aligned}
\]
Subtracting gives $(\lambda_k(b)-\lambda_\ell(b))\langle v^{(k)},v^{(\ell)}\rangle_W=0$, and since $\lambda_k(b)\ne \lambda_\ell(b)$ we conclude
$\langle v^{(k)},v^{(\ell)}\rangle_W=0$.
\end{proof}

\section{A Chebyshev identity and an explicit inverse}\label{sec:inverse}
Although the conjugation in Lemma~\ref{lem:similarity} involves $\sqrt{b}$, all quantities of arithmetic interest (determinants, minors, inverse entries) are rational functions of $b$. To make this transparent, we first relate the relevant Chebyshev values to repunits, and then rewrite the standard inverse formula for symmetric tridiagonal Toeplitz matrices in these terms.

We use the Chebyshev polynomials $U_m$ of the second kind, characterized by $U_0(x)=1$, $U_1(x)=2x$, and $U_{m+1}(x)=2x\,U_m(x)-U_{m-1}(x)$.

\begin{lemma}\label{lem:cheb}
Let $b>0$ and $x_b=(b+1)/(2\sqrt{b})$. Then for every $m\geqslant 0$,
\[
U_m(x_b)=b^{-m/2}\,R_{m+1}(b).
\]
\end{lemma}

\begin{proof}
If $b\ne 1$, set $z=\sqrt{b}$ so that $x_b=(z+z^{-1})/2$. The classical identity
\[
U_m\left(\frac{z+z^{-1}}{2}\right)=\frac{z^{m+1}-z^{-(m+1)}}{z-z^{-1}}
\qquad (z\ne \pm 1)
\]
gives
\[
U_m(x_b)=\frac{b^{(m+1)/2}-b^{-(m+1)/2}}{b^{1/2}-b^{-1/2}}
=b^{-m/2}\,\frac{b^{m+1}-1}{b-1}
=b^{-m/2}\,R_{m+1}(b).
\]
If $b=1$, then $x_b=1$ and the claim reduces to $U_m(1)=m+1=R_{m+1}(1)$.

\end{proof}

We now turn to the inverse. The starting point is the standard Chebyshev representation of the inverse of a symmetric tridiagonal Toeplitz matrix; Lemma~\ref{lem:cheb} eliminates all square roots and expresses the answer purely in terms of repunits.

\begin{theorem}[Inverse in terms of repunits]\label{thm:inverse}
Let $b>0$ and $n\geqslant 1$. Then $V_n(b)$ is invertible and for $1\leqslant i,j\leqslant n$ one has
\[
\left(V_n(b)^{-1}\right)_{i,j}
=
\frac{(-1)^{i+j}}{R_{n+1}(b)}
\begin{cases}
R_i(b)\,R_{n-j+1}(b), & i\leqslant j,\\
b^{\,i-j}\,R_j(b)\,R_{n-i+1}(b), & i\geqslant j.
\end{cases}
\]
\end{theorem}

\begin{proof}
By Proposition~\ref{prop:eigs}, all eigenvalues of $V_n(b)$ are positive, hence $V_n(b)$ is invertible.
By Lemma~\ref{lem:similarity} we have $V_n(b)^{-1}=D\,T_n(b)^{-1}\,D^{-1}$.
A standard explicit formula for the entries of $T_n(b)^{-1}$ (in terms of Chebyshev polynomials) is given in
\cite{SmithPDE,YuehCheng}. Combining that formula with Lemma~\ref{lem:cheb} yields, for $i\leqslant j$,
\[
\left(T_n(b)^{-1}\right)_{i,j}
=
(-1)^{i+j}\,b^{-(j-i)/2}\,\frac{R_i(b)\,R_{n-j+1}(b)}{R_{n+1}(b)},
\]
and symmetry gives the case $i\geqslant j$.
Conjugating by $D$ then gives the stated formula for $V_n(b)^{-1}$.
\end{proof}

\end{document}